\documentclass[10pt,letterpaper]{article}
\usepackage[utf8]{inputenc}
\usepackage[english]{babel}
\usepackage{amsmath}
\usepackage{amsfonts}
\usepackage{amssymb}
\usepackage{graphicx}
\usepackage{mathrsfs}

\usepackage{color}
\newcommand{\supp}{{\rm supp}}

\newcommand{\R}{{\mathbb{R}}}
\newcommand{\E}{{\mathbb{E}}}
\newcommand{\N}{{\mathbb{N}}}
\newcommand{\eps}{\varepsilon}  
\def\sideremark#1{\ifvmode\leavevmode\fi\vadjust{\vbox to0pt{\vss
 \hbox to 0pt{\hskip\hsize\hskip1em
 \vbox{\hsize2.1cm\tiny\raggedright\pretolerance10000
  \noindent #1\hfill}\hss}\vbox to15pt{\vfil}\vss}}}%

\usepackage{upref,amsthm,amsxtra,exscale}
\usepackage{cite}
\usepackage[colorlinks=true,urlcolor=blue,
citecolor=red,linkcolor=blue,linktocpage,pdfpagelabels,
bookmarksnumbered,bookmarksopen]{hyperref}

\newtheorem{theorem}{Theorem}[section]
\newtheorem{corollary}[theorem]{Corollary}
\newtheorem{lemma}[theorem]{Lemma}
\newtheorem{proposition}[theorem]{Proposition}
\newtheorem{remark}[theorem]{Remark}

\numberwithin{equation}{section}

\author{Mónica Clapp\footnote{M. Clapp was partially supported by UNAM-DGAPA-PAPIIT grant IN100718 (Mexico) and CONACYT grant A1-S-10457 (Mexico).} \quad and \quad Alberto Saldaña\footnote{A. Saldaña was supported by the Alexander von Humboldt Foundation (Germany).}}
\title{Entire nodal solutions to the critical Lane-Emden system}
\date{\today}

\begin{document}

\maketitle

\begin{abstract}
We establish the existence of finitely many sign-changing solutions to the Lane-Emden system 
$$-\Delta u=|v|^{q-2}v,\quad -\Delta v=|u|^{p-2}u \quad \text{ in }\R^N, \ \ N\geq 4,
$$
where the exponents $p$ and $q$ lie on the critical hyperbola $\frac{1}{p}+\frac{1}{q}=\frac{N-2}{N}$. These solutions are nonradial and arise as limit profiles of symmetric sign-changing minimizing sequences for a critical higher-order problem in a bounded domain.

\medskip

\medskip

\end{abstract}

\section{Introduction}
\label{sec:introduction}
Consider the Lane-Emden system
\begin{equation} \label{eq:system}
\begin{cases}
-\Delta u=|v|^{q-2}v,\\
-\Delta v=|u|^{p-2}u, \\
u\in D^{2,q'}(\mathbb{R}^{N}),\quad v\in D^{2,p'}(\mathbb{R}^{N}),
\end{cases}
\end{equation}
where $N\geq 3$ and $(p,q)$ lies on the critical hyperbola, that is,
\begin{equation} \label{eq:hyperbola}
\frac{1}{p}+\frac{1}{q}=\frac{N-2}{N}.
\end{equation} 
As usual, $p':=\frac{p}{p-1}$, $q':=\frac{q}{q-1}$, and $D^{2,r}(\mathbb{R}^{N})$ is the completion of $\mathcal{C}^\infty_c(\mathbb{R}^{N})$ with respect to the norm
$$\|w\|_r:=\left(\int_{\mathbb{R}^N}|\Delta w|^r\right)^\frac{1}{r}.$$

The reduction-by-inversion approach allows to reformulate the system \eqref{eq:system} as a higher-order quasilinear problem. Indeed, $(u,v)$ is a (strong) solution to \eqref{eq:system} if and only if $u$ is a (weak) solution of
\begin{equation} \label{eq:prob}
\begin{cases}
\Delta(|\Delta u|^{q'-2}\Delta u)=|u|^{p-2}u,\\
u\in D^{2,q'}(\mathbb{R}^{N})
\end{cases}
\end{equation}
and $v:=-|\Delta u|^{q'-2}\Delta u$; see Lemma \ref{lem:equivalence} below.  
  
Using a concentration-compactness argument, P.-L. Lions showed in \cite{l} that \eqref{eq:prob} has a positive solution when $(p,q)$ satisfies \eqref{eq:hyperbola}.  Thus, a \emph{positive} solution $(u,v)$ of  \eqref{eq:system}-\eqref{eq:hyperbola} exists.  Moreover, $u$ and $v$ are radially symmetric, and they are unique up to translations and dilations \cite{hv,CLO05}. This solution does not have, in general, an explicit formula like in the case of the scalar problem
\begin{equation} \label{eq:yamabe}
-\Delta u = |u|^{2^*-2}u,\qquad u\in D^{1,2}(\mathbb{R}^{N}),
\end{equation}
where $2^*:=\frac{2N}{N-2}$ is the critical Sobolev exponent; but the precise decay rates of $u$ and $v$ at infinity can be deduced and they depend in a subtle way on the value of the exponents $p$ and $q$; see \cite[Theorem 2]{hv}.

\medskip

In this paper, we establish the existence of \emph{sign-changing} solutions to \eqref{eq:system}-\eqref{eq:hyperbola}. Our main result is the following one. We use $\lfloor x \rfloor$ to denote the greatest integer less than or equal to $x$.

\begin{theorem} \label{thm:main}

If $(p,q)$ satisfies \eqref{eq:hyperbola}, then the system \eqref{eq:system} has at least $\lfloor \frac{N}{4}\rfloor$ nonradial sign-changing solutions, i.e., both components $u$ and $v$ change sign.
\end{theorem}

The solutions given by Theorem \ref{thm:main} have some explicit symmetries which provide some information on their shape; see Lemma \ref{lem:examples} and Remark \ref{rem:n3}.

Theorem \ref{thm:main} seems to be the first result regarding the existence of entire sign-changing solutions to \eqref{eq:system}-\eqref{eq:hyperbola}, except for the particular cases $p=q=2^*$ and $q=2$ (or $p=2$). 

When $p=q=2^*$, the solutions to \eqref{eq:system} are $(u,u)$, where $u$ is a solution to the Yamabe problem \eqref{eq:yamabe}, which is invariant under Möbius transformations. Taking advantage of this fact, W. Ding established the existence of infinitely many sign-changing solutions to \eqref{eq:yamabe} in  \cite{d}. They are invariant under the action of a group of conformal transformations whose orbits have positive dimension. Bubbling sign-changing solutions were obtained by del Pino, Musso, Pacard, and Pistoia in \cite{dmpp}, using the Lyapunov-Schmidt reduction method. Their solutions are different from those in \cite{d}.

When $q=2$, \eqref{eq:prob} becomes the Paneitz problem
\begin{equation} \label{eq:paneitz}
\Delta^2 u= |u|^{2_*-2}u,\qquad u\in D^{2,2}(\mathbb{R}^{N}),
\end{equation}
with $2_*:=\frac{2N}{N-4}$. Inspired by Ding's approach, Bartsch, Schneider and Weth \cite{bsw} established the existence of infinitely many solutions to \eqref{eq:paneitz} and to more general polyharmonic problems which, like \eqref{eq:paneitz}, are invariant under conformal transformations. We stress that the solutions given by Theorem \ref{thm:main} for $q=2$ are different from those in \cite{bsw}.

Unfortunately, the approach followed in \cite{d,bsw} does not apply to arbitrary $(p,q)$ on the critical hyperbola because, even though the problem  \eqref{eq:prob} is invariant under Euclidean transformations and dilations, it is not invariant under Möbius transformations in general; see Proposition \ref{prop:no_kelvin} below. On the other hand, the Lyapunov-Schmidt reduction method used in \cite{dmpp} relies on a good knowledge of the linearized problem and on the explicit form of the positive entire solution to \eqref{eq:yamabe}, but this information is not available for the Lane-Emden system \eqref{eq:system}. 

Yet another kind of sign-changing solutions to the Yamabe problem \eqref{eq:yamabe}, different from those in \cite{d,dmpp}, was recently discovered in \cite{c}. They arise as limit profiles of symmetric sign-changing minimizing sequences for the purely critical exponent problem in a bounded domain.   

To prove Theorem \ref{thm:main} we follow the strategy of \cite{c}, that is, we analyze the behavior of minimizing sequences, with a specific kind of symmetries, for the critical problem
\begin{equation} \label{eq:alternative}
\Delta(|\Delta u|^{q'-2}\Delta u)=|u|^{p-2}u,\qquad u\in D^{2,q'}_0(\Omega),
\end{equation}
when $\Omega$ is the unit ball. These symmetric functions, which we call $\phi$\emph{-equivariant}, may be chosen to be sign-changing by construction; see Section \ref{sec:profile}. Unlike the conformal symmetries considered in \cite{d,bsw} which prevent blow-up, our symmetries are given by linear isometries which have fixed points, thus allowing blow-up. We impose some conditions on the symmetries to ensure that the blow-up profile of $\phi$-equivariant minimizing sequences is $\phi$-equivariant; see assumptions $\textbf{(S1)}$ and $\textbf{(S2)}$ below.

There are two main sources of difficulties in performing the blow-up analysis: the nonlinear nature of the operator in equation \eqref{eq:prob} and the fact that it is of higher order. Like for the purely critical $p$-Laplacian problem  \cite{mw,cl}, it is delicate to show that the weak limit of a minimizing sequence for \eqref{eq:prob} solves a limit problem when $q\neq 2$. For the $p$-Laplacian this is usually achieved by using suitable truncations, but due to the higher-order nature of \eqref{eq:prob} this approach cannot be applied (a truncation may cause a jump discontinuity of the gradient, preventing the truncated function from being twice weakly differentiable).  We circumvent this difficulty using a more general approach based on mollifications. 

The concentration and blow-up behavior of $\phi$-equivariant minimizing sequences for \eqref{eq:alternative} in a bounded domain is given by Theorem \ref{thm:profile} below. This result contains an existence alternative: it asserts that there exists a $\phi$-equivariant minimizer for \eqref{eq:alternative}, either in the unit ball, or in a half-space, or in the whole space $\mathbb{R}^N$. Moreover, due to the presence of fixed points, these minimizers have the same energy; see Lemma \ref{lem:leastenergy}. Therefore, anyone of them is a $\phi$-equivariant least energy solution to the problem \eqref{eq:prob} in the whole space $\mathbb{R}^N$. We stress that, unlike for the Laplacian, a general unique continuation property is not available, as far as we know, for the problem \eqref{eq:alternative}. So one cannot discard the possibility of having solutions to \eqref{eq:prob} which vanish outside a ball or in a half-space.

Finally, we point out two limitations of our method. Firstly, it cannot be applied when $N=3$, because there are no groups in this dimension with the properties that we need; see Remark \ref{rem:n3}. Secondly, in contrast with the cases $q=2^*$ and $q=2$ considered in \cite{d,bsw}, our approach yields only finitely many solutions. The questions whether the system \eqref{eq:system} has a sign-changing solution in dimension $3$, or whether it has infinitely many solutions in every dimension, remain open.

\medskip

To close this introduction we mention some possible generalizations of Theorem \ref{thm:main}. For the sake of clarity, in this paper we have focused on the system \eqref{eq:system} and the associated higher-order problem \eqref{eq:prob}; but an inspection of the proofs shows that the same approach can be used to study the existence of finitely many entire nodal solutions to the Hardy-Littlewood-Sobolev system
\begin{equation*}
(-\Delta)^m u=|v|^{q-2}v,\quad
(-\Delta)^m v=|u|^{p-2}u, \quad 
u\in D^{2m,q'}(\mathbb{R}^{N}),\ v\in D^{2m,p'}(\mathbb{R}^{N}),
\end{equation*}
with $m\in \N$ and $\frac{1}{p}+\frac{1}{q}=\frac{N-2m}{N}$, or the associated higher-order problem
\begin{equation*}
\Delta^m(|\Delta^m u|^{q'-2}\Delta^m u)=|u|^{p-2}u,\quad 
u\in D^{2m,q'}(\mathbb{R}^{N}).
\end{equation*}
The left-hand side of the equation above can be regarded as a quasilinear version of the polyharmonic operator. Similarly, one could also consider the problem
\begin{equation*}
-\operatorname{div}(\Delta^m(|\nabla\Delta^m u|^{q'-2}\nabla\Delta^m u))=|u|^{p-2}u,\quad
u\in D^{2m+1,q'}(\mathbb{R}^{N}),
\end{equation*}
where $m\in \N\cup\{0\}$ and $\frac{1}{p}+\frac{1}{q}=\frac{N-2m-1}{N}$.  Note that this problem reduces to the critical $p$-Laplacian problem if $m=0$. The existence of finitely many sign-changing solutions, in this particular case, was shown in \cite{cl}.  The approach we present here can be used to extend Theorem \ref{thm:main} to anyone of these problems.

\medskip

The paper is organized as follows. In Section \ref{r:sec} we explain the reduction-by-inversion approach and show the equivalence between \eqref{eq:system} and \eqref{eq:prob}. In Section \ref{sec:profile} we introduce our symmetric variational framework and we give a precise description of the concentration and blow-up behavior of $\phi$-equivariant minimizing sequences for the higher-order problem \eqref{eq:alternative} in a bounded domain. Our main result, Theorem \ref{thm:main}, is proved in Section \ref{sec:existence}. Finally, in an appendix, we give conditions which guarantee that the weak limit of a minimizing sequence for the variational functional is a critical point of this functional.

\section{Reduction by inversion}\label{r:sec}

From now on, we assume that $(p,q)$ lies on the critical hyperbola \eqref{eq:hyperbola}. Then $p,q>\frac{N}{N-2}$ and $p=\frac{Nq'}{N-2q'}$, where $q':=\frac{q}{q-1}$. 

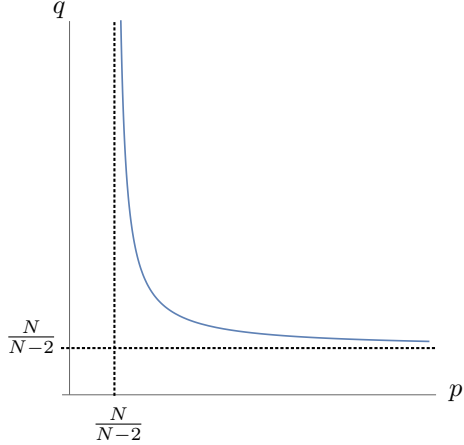
\begin{figure}[h!]
  \begin{center}
    \begin{picture}(150,160)
    \put(0,5){
    \includegraphics[height= 5cm]{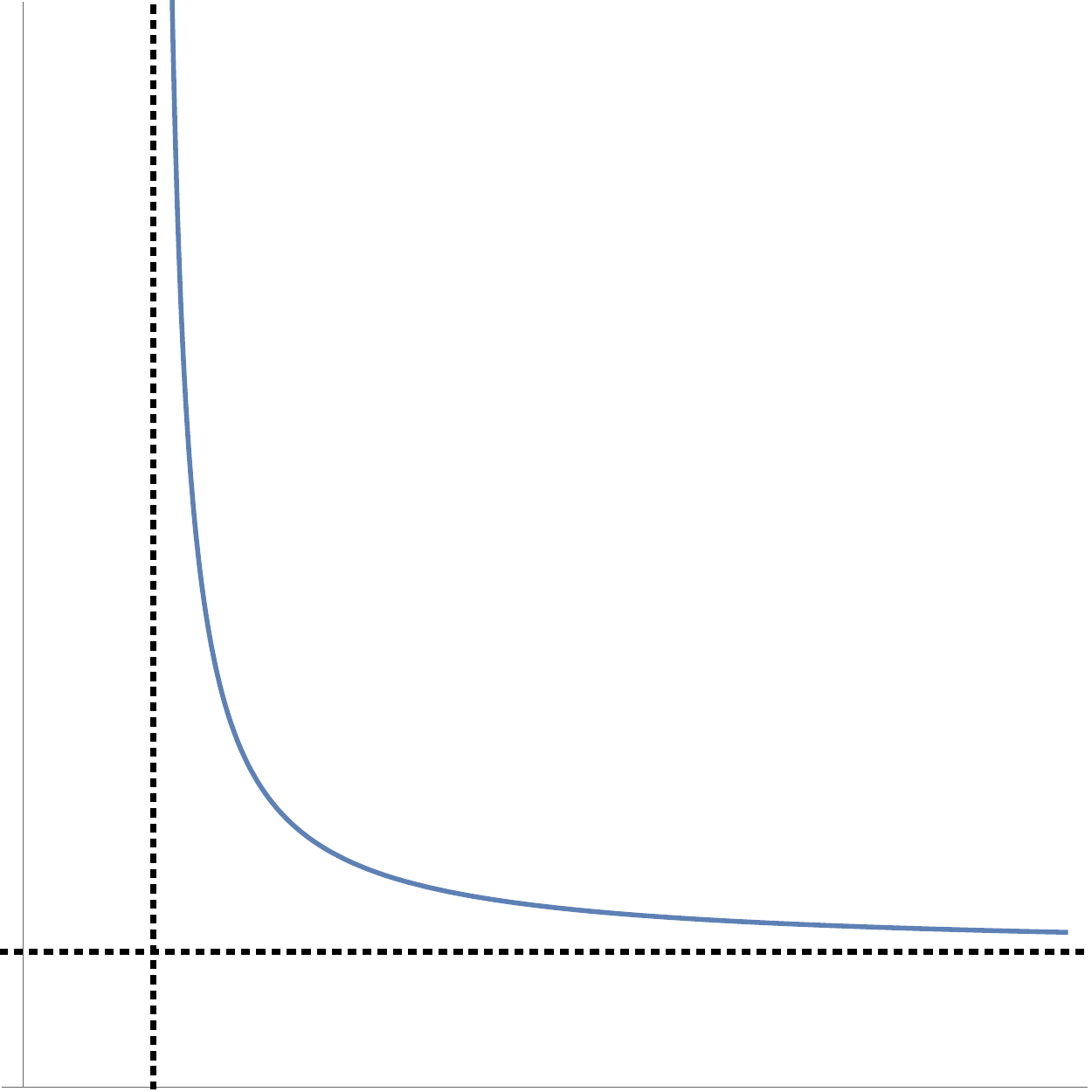}
    }
    \put(150,5){$p$}
    \put(0,150){$q$}
    \put(-18,25){$\frac{N}{N-2}$}
    \put(15,-8){$\frac{N}{N-2}$}
  \end{picture}
  \end{center}
\caption{The critical hyperbola and its asymptotes.}\label{hyp:fig}
\end{figure}

We consider the Banach space
\begin{align*}
D^{2,q'}(\mathbb{R}^N):=\{u\in L^p(\mathbb{R}^N):u\text{ is twice weakly differentiable},\;\Delta u\in L^{q'}(\mathbb{R}^N)\},
\end{align*}
endowed with the norm 
$$\|w\|_{q'}:=\left(\int_{\mathbb{R}^N}|\Delta w|^{q'}\right)^{\frac{1}{q'}}.$$
This space is the completion of $\mathcal{C}^\infty_c(\mathbb{R}^N)$ with respect to $\|\cdot\|_{q'}$, and $p$ is the critical exponent for the Sobolev embedding $D^{2,q'}(\mathbb{R}^N)\hookrightarrow L^p(\mathbb{R}^N)$.

A solution to the system \eqref{eq:system} is a critical point $(u,v)$ of the functional
$$I(u,v):=\int_{\mathbb{R}^N}\nabla u\cdot\nabla v -\frac{1}{p}\int_{\mathbb{R}^N}|u|^{p}-\frac{1}{q}\int_{\mathbb{R}^N}|v|^{q},$$
defined in the space $D^{1,(q')^*}(\mathbb{R}^N)\times D^{1,(p')^*}(\mathbb{R}^N)$, where $r^*:=\frac{Nr}{N-r}$. The critical points belong to $D^{2,q'}(\mathbb{R}^N)\times D^{2,p'}(\mathbb{R}^N)$ and are, therefore, strong solutions. 

By a solution $u$ to the problem \eqref{eq:prob} we mean a weak solution, i.e., a critical point of the functional
\begin{equation} \label{eq:j}
J(u):=\frac{1}{q'}\int_{\mathbb{R}^N}|\Delta u|^{q'}-\frac{1}{p}\int_{\mathbb{R}^N}|u|^{p},
\end{equation}
defined in the Banach space $D^{2,q'}(\mathbb{R}^N)$. Its derivative is given by
\begin{align*}
J'(u)\varphi=\int_{\mathbb{R}^N} |\Delta u|^{q'-2}\Delta u\Delta \varphi - \int_{\mathbb{R}^N}|u|^{p-2}u\varphi.
\end{align*} 

The following lemma establishes the equivalence between solutions to the higher-order quasilinear problem \eqref{eq:prob} and to the system \eqref{eq:system}. This relation is sometimes called \emph{reduction-by-inversion}. We refer to the surveys \cite{bdt,df,r} for an overview of the diversity of methods used in the study of \eqref{eq:system} and more general Hamiltonian systems. 

\begin{lemma} \label{lem:equivalence}
$u$ is a solution of \eqref{eq:prob} and $v=-|\Delta u|^{q'-2}\Delta u$ iff $(u,v)$ is a solution of \eqref{eq:system}.
\end{lemma}

\begin{proof}
Let $u\in D^{2,q'}(\mathbb{R}^N)$ be a solution of \eqref{eq:prob} and set $v:=-|\Delta u|^{q'-2}\Delta u$. It is easy to see that $v$ coincides  a.e. in $\mathbb{R}^N$ with the (unique) solution to the problem $-\Delta w=|u|^{p-2}u$, $w\in D^{1,p'}(\mathbb{R}^N)$, which belongs to $D^{2,p'}(\mathbb{R}^N)$. Therefore, $v\in D^{2,p'}(\mathbb{R}^N)$. Note that $|v|^{q-2}v=-\Delta u$. Hence, for every $\varphi\in C^\infty_c(\mathbb{R}^N)$, we have that
\begin{align*}
\partial_uI(u,v)\varphi&= \int_{\mathbb{R}^N}\nabla v\cdot\nabla\varphi - \int_{\mathbb{R}^N}|u|^{p-2}u\varphi = \int_{\mathbb{R}^N}-v\,(\Delta\varphi) - \int_{\mathbb{R}^N}|u|^{p-2}u\varphi \\
&=J'(u)\varphi=0,\\
\partial_vI(u,v)\varphi&= \int_{\mathbb{R}^N}\nabla u\cdot\nabla\varphi - \int_{\mathbb{R}^N}|v|^{q-2}v\varphi = \int_{\mathbb{R}^N}-(\Delta u)\,\varphi - \int_{\mathbb{R}^N}|v|^{q-2}v\varphi=0,
\end{align*}
i.e., $(u,v)$ solves \eqref{eq:system}. The converse is proved in a similar way. 
\end{proof}

\section{Sign-changing minimizing sequences}
\label{sec:profile}

To produce sign-changing solutions we introduce suitable symmetries, as in \cite{c,cl}. 

Let $G$ be a closed subgroup of the group $O(N)$ of linear isometries of $\mathbb{R}^{N}$ and let $\phi:G\rightarrow\mathbb{Z}_{2}:=\{1,-1\}$ be a continuous homomorphism of groups. We write $Gx:=\{gx:g\in G\}$ for the $G$-orbit of a point $x\in\mathbb{R}^{N}$. From now on, we assume that $G$ and $\phi$ have the following properties: 
\begin{itemize}
\item[$(\mathbf{S1})$]For each $x \in \mathbb{R}^N$, either $\dim(Gx)>0$ or $Gx=\{x\}$.
\item[$(\mathbf{S2})$]There exists $\xi\in\mathbb{R}^{N}$ such that $\{g\in G:g\xi=\xi\}\subset\ker\phi$.
\end{itemize} 

Now, let $\Omega$ be a $G$-invariant domain in $\mathbb{R}^N$, i.e., $Gx\subset\Omega$ for every $x\in\Omega$. A function $u:\Omega\to\mathbb{R}$ is called $\phi$\emph{-equivariant} if
\begin{equation*}
u(gx)=\phi(g)u(x)\qquad \text{for all}\,\, g\in G,\,x\in\Omega.
\end{equation*}
If $\phi\equiv 1$, then a $\phi$-equivariant function is simply a $G$-invariant function. On the other hand, if $\phi:G\to\mathbb{Z}_{2}$ is surjective, then every nontrivial $\phi$-equivariant function is nonradial and changes sign.

We denote the closure of $\mathcal{C}_c^\infty(\Omega)$ in $D^{2,q'}(\mathbb{R}^N)$ by $D_0^{2,q'}(\Omega)$, and set
\begin{equation*}
D_0^{2,q'}(\Omega)^\phi:=\{u\in D_0^{2,q'}(\Omega):u\text{ is }\phi\text{-equivariant} \}.
\end{equation*}
We say that $u$ is a solution of
\begin{equation} \label{eq:prob_omega}
\begin{cases}
\Delta(|\Delta|^{q'-2}\Delta u)=|u|^{p-2}u,\\
u\in D^{2,q'}_0(\Omega),
\end{cases}
\end{equation}
if $u$ is a critical point of the $\mathcal{C}^1$-functional $J:D_0^{2,q'}(\Omega)\to \mathbb{R}$ given by \eqref{eq:j}. 

Let $\mathcal{C}_c^{\infty}(\Omega)^{\phi}:=\{u\in \mathcal{C}_c^{\infty}(\Omega):u\text{ is }\phi\text{-equivariant}\}$. Given $\varphi\in\mathcal{C}_{c}^{\infty}(\Omega)$, we define
\begin{align}\label{phi:notation}
\varphi_\phi(x):=\frac{1}{\mu(G)}\int_{G}\phi(g)\varphi(gx)\,\mathrm{d}\mu(g),  
\end{align}
where $\mu$ is the Haar measure on $G$. Then $\varphi_\phi\in\mathcal{C}_c^{\infty}(\Omega)^{\phi}$.

\begin{lemma} \label{lem:symmetric_criticality}
If $u\in D^{2,q'}_0(\Omega)^{\phi}$, then 
$$J'(u)[\varphi_\phi]=J'(u)\varphi\qquad \text{for every }\varphi\in\mathcal{C}_{c}^{\infty}(\Omega).$$ Consequently, if $J'(u)\vartheta = 0$ for every $\vartheta \in \mathcal{C}_c^{\infty}(\Omega)^{\phi}$, then $J'(u)\varphi = 0$ for every $\varphi \in \mathcal{C}_c^{\infty}(\Omega)$, i.e., $u$ is a solution to the problem \eqref{eq:prob_omega}.
\end{lemma}

\begin{proof}
Note that $\Delta(v\circ g)=(\Delta v)\circ g$ for every $v\in D^{2,q'}_0(\Omega)$, $g\in G$. So, as $u$ is $\phi$-equivariant, we have that $\Delta u$ is $\phi$-equivariant. Also,
$$\Delta(\varphi_\phi)(x)=\frac{1}{\mu(G)}\int_{G}\Delta(\phi(g)\varphi\circ g)(x)\,\mathrm{d}\mu(g)=\frac{1}{\mu(G)}\int_{G}\phi(g)\,\Delta\varphi(gx)\,\mathrm{d}\mu(g).$$
Fubini's theorem and a suitable change of variables yield 
\begin{equation*}
\int_{\Omega}|\Delta u|^{q'-2}\Delta u\,\Delta(\varphi_\phi) = \int_{\Omega}|\Delta u|^{q'-2}\Delta u\, \Delta\varphi,\quad\int_{\Omega}|u|^{p-2}u(\varphi_\phi) = \int_{\Omega} |u|^{p-2}u\varphi.
\end{equation*}
Hence, $J'(u)[\varphi_\phi]=J'(u)\varphi$, as claimed.
\end{proof}

The nontrivial $\phi$-equivariant solutions to \eqref{eq:prob_omega} belong to the set
$$\mathcal{N}^\phi(\Omega):=\{u\in D_0^{2,q'}(\Omega)^\phi:u\neq0,\,\|u\|_{q'}^{q'}=|u|_{p}^{p}\},$$
where $|\cdot|_{p}$ denotes the norm in $L^p(\mathbb{R}^N)$. Define
$$c^\phi(\Omega):=\inf_{u\in \mathcal{N}^\phi(\Omega)}J(u).$$
Property $(\mathbf{S2})$ guarantees that the space $D_0^{2,q'}(\Omega)^\phi$ is infinite dimensional; see \cite{bcm}. Therefore, $\mathcal{N}^\phi(\Omega)\neq\emptyset$ and $c^\phi(\Omega)\in\mathbb{R}$. Note that $(\mathbf{S2})$ is trivially satisfied if $\phi\equiv 1$.

Next we describe the limit profile of minimizing sequences for the functional $J$ on $\mathcal{N}^\phi(\Omega)$. As we shall see, property $(\mathbf{S1})$ guarantees that the limit profile is $\phi$-equivariant; see Theorem \ref{thm:profile}. So it will be sign-changing if $\phi$ is surjective.

We start by listing some properties of $\mathcal{N}^\phi(\Omega)$.

\begin{lemma} \label{lem:nehari}
\begin{enumerate}
	\item[$(a)$]There exists $a_0 >0$ such that $\|u\|_{q'}\geq a_0$ for every $u\in \mathcal{N}^\phi(\Omega)$.
	\item[$(b)$]$\mathcal{N}^\phi(\Omega)$ is a $\mathcal{C}^1$-Banach submanifold of $D_0^{2,q'}(\Omega)^\phi$, and a natural constraint for $J$. 
	\item[$(c)$]Let $\mathscr{T}:=\left\{\sigma \in \mathcal{C}^0\left([0,1],D_0^{2,q'}(\Omega)^\phi\right) :\sigma(0)=0,\,\sigma(1)\neq 0, \,J(\sigma(1)) \leq 0\right\}.$ Then,
	$$c^\phi(\Omega)=\inf_{\sigma \in \mathscr{T}} \max_{t \in [0,1]}J(\sigma(t)).$$
\end{enumerate}
\end{lemma}

\begin{proof}
The proof is similar to that of \cite[Lemma 2.1]{cl} and we omit it.
\end{proof}

As usual, let $$\Omega^G:=\{x\in \Omega:Gx=\{x\}\}$$ denote the set of $G$-fixed points in $\Omega$. The proof of the next lemma is similar to that of \cite[Lemma 2.3]{cl}. We include it here for the sake of completeness.

\begin{lemma}
\label{lem:leastenergy}
If $\Omega$ is a $G$-invariant smooth domain in $\mathbb{R}^N$ and $\Omega^G \neq \emptyset$, then
$$c^\phi(\Omega)=c^\phi(\mathbb{R}^N)=:c_{\infty}^\phi.$$
\end{lemma}

\begin{proof}
As 
\begin{align}\label{emded:eq}
D_0^{2,q'}(\Omega)^\phi\subset D^{2,q'}(\mathbb{R}^N)^\phi,  
\end{align}
one has that $c_{\infty}^\phi \leq c^\phi(\Omega)$. For the opposite inequality, we use the translation and dilation invariance of the problem \eqref{eq:prob}. Fix $x_0\in \Omega^G$ and let $(\varphi_k)$ be a sequence in $\mathcal{N}^\phi(\mathbb{R}^N) \cap \mathcal{C}_c^\infty(\mathbb{R}^N)$ such that $J(\varphi_k) \to c_\infty^\phi$. Since $\varphi_k$ has compact support, we may choose $\varepsilon_k >0$ such that the support of $\tilde{\varphi}_k(x) := \varepsilon_k^{-N/p} \varphi_k (\varepsilon_k^{-1}(x-x_0))$ is contained in $\Omega$. As $x_0$ is a $G$-fixed point, $\tilde{\varphi}_k$ is $\phi$-equivariant. Thus, $\tilde{\varphi}_k \in \mathcal{N}^\phi(\Omega)$ and, hence,
  \begin{equation*}
    c^\phi(\Omega) \le J(\tilde{\varphi}_k) = J(\varphi_k) \quad \text{for all } k.
  \end{equation*}
Letting $k \to \infty$ we conclude that $c^\phi(\Omega) \leq c_{\infty}^\phi$.
\end{proof}

\begin{lemma}
\label{lem:orbits}
If $G$ satisfies $(\mathbf{S1})$ then, for every pair of sequences $(\varepsilon_k)$ in $(0,\infty)$ and $(x_k)$ in $\mathbb{R}^N$, there exists a positive constant $C_0>0$ and a sequence $(\xi_k)$ in $\mathbb{R}^N$ such that, after passing to a subsequence,
\begin{equation}
\label{eq:orbits}
\varepsilon_k^{-1}\mathrm{dist}(Gx_k,\xi_k) \leq C_0 \qquad\text{for all }\,k.
\end{equation}
 Furthermore, one of the following statements holds true:
\begin{enumerate}
	\item[$(a)$]either $\xi_k \in (\mathbb{R}^N)^G$ for all $k\in\N$,
	\item[$(b)$]or, for each $m \in \mathbb{N}$, there exist $g_1,\ldots,g_m \in G$ such that
	$$\varepsilon_k^{-1} |g_i\xi_k - g_j\xi_k| \to \infty \quad\text{as }\,k \to \infty \qquad\text{if }\,i \neq j.$$
\end{enumerate}
\end{lemma}
\begin{proof}
A detailed proof can be found in \cite[Lemma 2.4]{cl}. Here we just give a sketch. For each $k\in\N$ we write $x_k=z_k+y_k$, where $z_k\in (\R^N)^G$ and $y_k\in [(\R^N)^G]^\bot$.  If $(\eps_k^{-1} y_k)$ does not contain a bounded subsequence, an easy argument shows that $\xi_k:= x_k$ satisfies $(b)$ and \eqref{eq:orbits}. On the other hand, if $(\eps_k^{-1} y_k)$ has a bounded subsequence, then $\xi_k:=z_k$ satisfies $(a)$ and \eqref{eq:orbits}.
\end{proof}

The next theorem is our main concentration result. The proof follows the general strategy of \cite[Theorem 8.13]{w} and \cite[Theorem 2.5]{cl}, but some parts require a different and more careful treatment.

\begin{theorem}
\label{thm:profile}
Assume $(\mathbf{S1})$ and $(\mathbf{S2})$. Let $\Omega$ be a $G$-invariant bounded smooth domain in $\mathbb{R}^N$ and $(u_k)$ be a sequence such that
$$u_k\in D^{2,q'}_0(\Omega)^{\phi},\quad J(u_k) \to c^{\phi}(\Omega),\quad \text{and }\quad J'(u_k) \to 0\, \text{ in }\,(D^{2,q'}_0(\Omega)^{\phi})'.$$ 
Then, up to a subsequence, one of the following two possibilities occurs: 
\begin{enumerate}
\item[\emph{(I)}]either $(u_k)$ converges strongly in $D_0^{2,q'}(\Omega)$ to a minimizer of $J$ on $\mathcal{N}^\phi(\Omega)$,
\item[\emph{(II)}]or there exist a sequence of $G$-fixed points $(\xi_k)$ in $\mathbb{R}^N$, a sequence $(\varepsilon_k) \in (0,\infty)$, and a nontrivial solution $w$ to the problem
\begin{equation} \label{Hprob}
\Delta(|\Delta w|^{q'-2}\Delta w)=|w|^{p-2}w, \qquad w\in D_0^{2,q'}(\mathbb{E})^\phi,
\end{equation}
with the following properties:
	\begin{enumerate}
	\item[\emph{(i)}]$\varepsilon_k \to 0$,\, $\xi_k \to \xi$, \, $\xi\in (\bar{\Omega})^G$,\, and\, $\varepsilon_k^{-1}\mathrm{dist}(\xi_k,\Omega)\to d\in [0,\infty]$.
	\item[\emph{(ii)}]If $d=\infty$, then $\mathbb{E} = \mathbb{R}^N$ and \,$\xi_k \in \Omega$. 
	\item[\emph{(iii)}]If $d\in [0,\infty)$, then $\xi \in \partial \Omega$ and $\mathbb{E} = \{ x \in \mathbb{R}^N: x \cdot \nu > \bar d  \}$, where $\nu$ is the inward pointing unit normal to $\partial \Omega$ at $\xi$ and $\bar d \in \{ d,-d \}$. Moreover, $\mathbb E$ is $G$-invariant, $\mathbb{E}^G\neq\emptyset$, and $\Omega^G\neq\emptyset$.
	\item[\emph{(iv)}]$w\in\mathcal{N}^\phi(\mathbb{E})$ and $J(w)=c_{\infty}^{\phi}$. 
	\item[\emph{(v)}]$\lim\limits_{k\to\infty}\left\| u_k-\varepsilon_k^{-N/p} w\left( \frac{x-\xi_k}{\varepsilon_k}  \right) \right\|=0$.
	\end{enumerate}
\end{enumerate}
\end{theorem}

\begin{proof}
As $q'>1$ and
\begin{equation}
\label{eq:bounded}
\frac{2}{N}\|u_k\|_{q'}^{q'} = J(u_k) - \frac{1}{p}J'(u_k)u_k \leq C+o(1)\|u_k\|,
\end{equation}
the sequence $(u_k)$ is bounded in $D_0^{2,q'}(\Omega)^{\phi}$ and, after passing to a subsequence, there is $u\in D_0^{2,q'}(\Omega)^{\phi}$ such that $u_k\rightharpoonup u$ weakly in $D_0^{2,q'}(\Omega)^{\phi}$. Hence, by Proposition \ref{prop:A}, $J'(u)\varphi=0$ for all $\varphi\in C^\infty_c(\Omega)^\phi$. We consider two cases:

(I) If $u\neq 0$, then $u\in \mathcal{N}^{\phi}(\Omega)$. So, from \eqref{eq:bounded} and our assumptions, we get
$$c^{\phi}(\Omega) \leq J(u) = \frac{2}{N}\|u\|_{q'}^{q'}\leq \liminf_{k\to\infty}\frac{2}{N}\|u_k\|_{q'}^{q'} = c^{\phi}(\Omega)+o(1).$$
Hence, $u_k \to u$ strongly in $D_0^{2,q'}(\Omega)^{\phi}$ and $J(u)=c^{\phi}(\Omega)$.

(II) Assume that $u=0$. Fix $\delta \in (0,\frac{N}{4}c^\phi_\infty)$, where $c^\phi_\infty:=c^\phi(\mathbb{R}^N)\leq c^\phi(\Omega)$. Since 
$$\int_{\Omega}|u_k|^{p} = \frac{N}{2}\left(J(u_k) - \frac{1}{q'}J'(u_k)u_k\right) \to \frac{N}{2} c^{\phi}(\Omega),$$
there are bounded sequences $(\varepsilon_k)$ in $(0,\infty)$ and $(x_k)$ in $\mathbb{R}^N$ such that, after passing to a subsequence,
\begin{equation*}
\delta = \sup_{x\in\mathbb{R}^N}\int_{B_{\varepsilon_k}(x)}|u_k|^{p}=\int_{B_{\varepsilon_k}(x_k)}|u_k|^{p},
\end{equation*}
where $B_r(x):=\{z\in \mathbb{R}^N:|z-x|<r\}$. For this choice of $(x_k)$ and $(\eps_k)$ let $C_0>0$ and $(\xi_k)$ as in Lemma \ref{lem:orbits}. Then, $|g_k x_k-\xi_k| \leq C_0\varepsilon_k$ for some $g_k \in G$ and, as $|u_k|$ is $G$-invariant, setting $C_1:=C_0+1$, we have that
\begin{equation}
\label{eq:delta}
\delta=\int_{B_{\varepsilon_k}(g_k x_k)}|u_k|^{p} \leq \int_{B_{C_1\varepsilon_k}(\xi_k)}|u_k|^{p}.
\end{equation}
This implies, in particular, that
\begin{equation}
\label{eq:dist}
\mathrm{dist}(\xi_k,\Omega) \leq C_1\varepsilon_k.
\end{equation}
We claim that $\xi_k \in (\mathbb{R}^N)^G$. Otherwise, for each $m\in \mathbb{N}$, Lemma \ref{lem:orbits} yields $m$ elements $g_1,\ldots,g_m \in G$ such that $B_{C_1\varepsilon_k}(g_i \xi_k) \cap  B_{C_1\varepsilon_k}(g_j \xi_k)=\emptyset$ if $i \neq j$, for $k$ large enough, and from \eqref{eq:delta} we would get that
$$m\delta \leq \sum\limits_{i=1}^m \int_{B_{C_1\varepsilon_k}(g_i\xi_k)}|u_k|^{p} \leq \int_{\Omega}|u_k|^{p} = \frac{N}{2} c^{\phi}(\Omega) + o(1),$$
for every $m\in \mathbb{N}$, which is a contradiction. This proves that $\xi_k \in (\mathbb{R}^N)^G$.

Define
\begin{equation*}
  w_k(y):=\varepsilon_k^{N/p}u_k(\varepsilon_ky + \xi_k)\qquad \text{for }\,y \in \Omega_k:=\{y\in \mathbb{R}^N: \varepsilon_ky + \xi_k \in \Omega \}.
\end{equation*}
Since $u_k$ is $\phi$-equivariant and $\xi_k$ is a $G$-fixed point, we have that $w_k$ is $\phi$-equivariant. Note that 
\begin{align}\label{star1}
\delta=\sup_{z\in\R^N}\int_{B_{1}(z)}|w_k|^{p}\leq \int_{B_{C_1}(0)}|w_k|^{p}.
\end{align}
Moreover, $(w_k)$ is a bounded sequence in $D^{2,q'}(\mathbb{R}^N)$. Hence, there is $w\in D^{2,q'}(\mathbb{R}^N)^{\phi}$ such that, up to a subsequence, $w_k \rightharpoonup w$ weakly in $D^{2,q'}(\mathbb{R}^N)^{\phi}$, and, by the Rellich-Kondrashov theorem, 
\begin{align}\label{star2}
|\nabla w_k|\to|\nabla w|\quad\text{and}\quad w_k \to w\quad\text{ in }\,L^{q'}_{\mathrm{loc}}(\mathbb{R}^N)\,\text{ and a.e. in } \mathbb{R}^N.
\end{align}
We claim that $w\neq 0$. To prove this, first note that, for any given $\varphi\in\mathcal{C}_{c}^\infty(\mathbb{R}^N)$, if we set
$$\vartheta(x):=\frac{1}{\mu(G)}\int_{G}\varphi(gx)\,\mathrm{d}\mu(g)\qquad\text{and}\qquad \vartheta_k(x):=\vartheta\left(\frac{x-\xi_k}{\varepsilon_k}\right),$$
then, using the notation in \eqref{phi:notation}, we have that $(\varphi w_k)_\phi= \vartheta w_k$, \,$\vartheta_{k} u_{k}$ is $\phi$-equivariant, and the sequence $(\vartheta_{k} u_{k})$ is bounded in $D_{0}^{2,q'}(\Omega)^\phi$. So Lemma~\ref{lem:symmetric_criticality} and a suitable rescaling yield
\begin{equation} \label{star3}
J'(w_k)[\varphi w_k]=J'(w_k)[\vartheta w_k]=J'(u_k)[\vartheta_ku_k]=o(1),
\end{equation}
because $J'(u_k) \to 0$ in $(D^{2,q'}_0(\Omega)^{\phi})'$. Now, arguing by contradiction, assume that $w=0$. Then, for any $\varphi\in \mathcal{C}_c^\infty(B_1(z))$ with $z\in\mathbb{R}^N$, we have by \eqref{star1}, \eqref{star2}, \eqref{star3}, Hölder's inequality, and Sobolev's inequality, that
\begin{align*}
\int_{\mathbb{R}^N}|\Delta(\varphi^2 w_k)|^{q'} &= \int_{\mathbb{R}^N}\varphi^{2q'}|\Delta w_k|^{q'} + o(1) \\
&= \int_{\mathbb{R}^N}|\Delta w_k|^{q'-2}\Delta w_k\,\Delta(\varphi^{2q'} w_k) + o(1)\\
&= \int_{\mathbb{R}^N}|w_k|^{p-2}w_k\,(\varphi^{2q'} w_k) + o(1)\\
&= \int_{\mathbb{R}^N}|w_k|^{p-q'}|\varphi^2 w_k|^{q'} + o(1) \\
&\leq \left(\int_{B_1(z)}|w_k|^p\right)^{(p-q')/p} \left(\int_{\mathbb{R}^N}|\varphi^2 w_k|^p\right)^{q'/p} + o(1) \\
&\leq \delta^{(p-q')/p}\,\frac{2}{N}(c_\infty^\phi)^{(q'-p)/p}\int_{\mathbb{R}^N}|\Delta(\varphi^2 w_k)|^{q'}+o(1) \\
&\leq \frac{1}{2}\int_{\mathbb{R}^N}|\Delta(\varphi^2 w_k)|^{q'}+o(1).
\end{align*}
Therefore, $\|\varphi^2 w_k\|_{q'}=o(1)$ and, hence, $|\varphi^2 w_k|_p=o(1)$ for any $\varphi\in \mathcal{C}_c^\infty(B_1(z))$, $z\in\mathbb{R}^N$. It follows that $w_k\to 0$ in $L^p_{\mathrm{loc}}(\mathbb{R}^N)$, contradicting \eqref{star1}.

Since $u_k\rightharpoonup 0$ and $w_k\rightharpoonup w\neq0$ we deduce that $\eps_k\to 0$. Moreover, passing to a subsequence, we have that $\xi_k \to \xi \in (\mathbb{R}^N)^G$.  Let 
\begin{equation*}
  d:=\lim_{k\to\infty}\varepsilon_k^{-1} \mathrm{dist}(\xi_k,\partial\Omega) \in [0,\infty].
\end{equation*}
We consider two cases:
\begin{itemize}
\item[(a)] If $d=\infty$ then, by \eqref{eq:dist}, we have that $\xi_k \in\Omega$. Hence, for every compact subset $X$ of $\mathbb{R}^N$, there exists $k_0$ such that $X\subset \Omega_k$ for all $k\ge k_0$. In this case, we set $\mathbb{E}:=\mathbb{R}^N$.
\item[(b)] If $d \in [0,\infty)$ then, as $\varepsilon_k \to 0$, we have that $\xi \in \partial \Omega$. If a subsequence of $(\xi_k)$ is contained in $\bar \Omega$ we set $\bar d:=-d$, otherwise we set $\bar d:=d$. We consider the half-space 
  \begin{equation*}
   \mathbb{H} :=\{ y\in\mathbb{R}^N: y\cdot\nu > \bar d \},
  \end{equation*}
where $\nu$ is the inward pointing unit normal to $\partial \Omega$ at $\xi$. Since $\xi$ is a $G$-fixed point, so is $\nu$. Thus, $\Omega^G \neq\emptyset$, $\mathbb{H}$ is $G$-invariant and $\mathbb{H}^G \neq \emptyset$. It is easy to see that, if $X$ is compact and $X \subset \mathbb{H}$, there exists $k_0$ such that $X\subset \Omega_k$ for all $k\ge k_0$. Moreover, if  $X$ is compact and $X \subset \mathbb{R}^N \smallsetminus \bar{\mathbb{H}}$, then $X\subset \mathbb{R}^N \smallsetminus\Omega_k$ for $k$ large enough. As $w_k \to w$ a.e. in $\mathbb{R}^N$, this implies that $w=0$ a.e. in $\mathbb{R}^N \smallsetminus \mathbb{H}$. So $w \in D_0^{2,q'}(\mathbb{H})^{\phi}$. In this case, we set $\mathbb{E}:=\mathbb{H}$.
\end{itemize}

For $\varphi\in C^\infty_c(\E)^\phi$ and  $\psi\in C^\infty_c(\E)^G$ define
$$\varphi_k(x):=\eps_k^{-N/p}\varphi\left(\frac{x-\xi_k}{\eps_k}\right), \qquad \psi_k(x):=\eps_k^{-N/p}(\psi(w_k-w))\left(\frac{x-\xi_k}{\eps_k}\right).$$
Then $\varphi_k$ and $\psi_k$ are $\phi$-equivariant.  Observe that $\supp(\varphi_k)$ and $\supp(\psi_k)$ are contained in $\Omega$ for $k$ sufficiently large, and $(\varphi_k)$ and $(\psi_k)$ are bounded in $D^{2,q'}_0(\Omega)$. Therefore,
\begin{align*}
J'(w_k)\varphi =J'(u_k)\varphi_k =o(1),\qquad J'(w_k)[\psi(w_k-w)] =J'(u_k)\psi_k=o(1).
\end{align*}
Then, by Proposition \ref{prop:A}, $w$ is a nontrivial solution of \eqref{Hprob}. Lemma \ref{lem:leastenergy} asserts that $c^\phi(\Omega)=c^\phi(\E)=c^\phi_\infty$.  Hence,
\begin{align*}
  c^\phi_\infty\leq \frac{2}{N}\|w\|_{q'}^{q'}
\leq \liminf_{k\to\infty}\frac{2}{N}\|w_k\|_{q'}^{q'}
=\frac{2}{N}\liminf_{k\to\infty}\|u_k\|_{q'}^{q'}=c^\phi_\infty.
\end{align*}
This implies that $J(w)=c^\phi_\infty$ and $w_k\to w$ strongly in $D^{2,q'}(\R^N)$. Consequently,
\begin{align*}
o(1)=\|w_k-w\|_{q'}^{q'}=\left\|w_k-\eps_k^{-N/p}w\left(\frac{x-\xi_k}{\eps_k}\right)\right\|_{q'}^{q'},
\end{align*}
and the proof is complete.
\end{proof}

An immediate consequence of the previous theorem is the following existence result.

\begin{corollary} \label{cor:bounded}
Assume that $G$ and $\phi$ satisfy $(\mathbf{S1})$ and $(\mathbf{S2})$, and let $\Omega$ be a $G$-invariant bounded smooth domain in $\mathbb{R}^N$ such that $\Omega^G=\emptyset$. Then the problem
\begin{equation} \label{eq:phi_bounded}
\begin{cases}
\Delta(|\Delta|^{q'-2}\Delta u)=|u|^{p-2}u,\\
u\in D^{2,q'}_0(\Omega)^\phi,
\end{cases}
\end{equation}
has a least energy solution. This solution is sign-changing if $\phi:G\to\mathbb{Z}_2$ is surjective. 
\end{corollary}

\begin{proof}
By statements $(a)$ and $(c)$ of Lemma \ref{lem:nehari}, and \cite[Theorem 2.9]{w}, there exists a sequence  $(u_k)$ such that
$$u_k\in D^{2,q'}_0(\Omega)^{\phi},\quad J(u_k) \to c^{\phi}_\infty,\quad \text{and }\quad J'(u_k) \to 0\, \text{ in }\,(D^{2,q'}_0(\Omega)^{\phi})'.$$ 
As $\Omega$ does not contain $G$-fixed points, the statement (II) in Theorem \ref{thm:profile} cannot hold true. Hence, $J$ attains its minimum on $\mathcal{N}^\phi(\Omega)$.
\end{proof}

In fact, arguing as in \cite[Corollary 3.2]{cf}, one should be able to prove that, under the assumptions of Corollary \ref{cor:bounded}, problem \eqref{eq:phi_bounded} has an unbounded sequence of solutions.

Note that the solution $u$ given by Corollary \ref{cor:bounded} does not yield a solution of the Dirichlet system 
\begin{align*}
   -\Delta u=|v|^{q-2}v,\quad -\Delta v=|u|^{p-2}u\quad \text{ in }\Omega,\qquad 
    u=v=0\quad \text{ on }\partial \Omega,
\end{align*}
due to the incompatibility of the boundary conditions. To obtain a solution to this system we would need to replace $D^{2,q'}_0(\Omega)$ with the Navier space $Y(\Omega)=D^{2,q'}(\Omega)\cap D^{1,q'}_0(\Omega)$ in problem \eqref{eq:phi_bounded}; see, e.g., \cite[Section 4]{bdt}. Observe, however, that there is no energy-preserving embedding of $Y(\Omega)$ into $Y(\R^N)$, and this is an important property required in our method; see \eqref{emded:eq}.

\section{Entire nodal solutions}
\label{sec:existence}

In this section we prove our main theorem. We start with the following existence result.

\begin{theorem}
\label{thm:existence}
Let $G$ be a closed subgroup of $O(N)$ and $\phi:G\rightarrow\mathbb{Z}_{2}$ be a continuous homomorphism satisfying $\mathbf{(S1)}$ and $\mathbf{(S2)}$. Then $J$ attains its minimum on $\mathcal{N}^\phi(\mathbb{R}^N)$. Consequently, the problem \eqref{eq:prob} has a nontrivial $\phi$-equivariant solution. This solution is sign-changing if $\phi:G\to\mathbb{Z}_2$ is surjective.
\end{theorem}

\begin{proof}
The unit ball $\mathbb{B}=\{x\in \mathbb{R}^N:|x|<1\}$ is $G$-invariant for every subgroup $G$ of $O(N)$. Note that, as $0\in\mathbb{B}^G$, we have that $c^\phi(\mathbb{B})=c_{\infty}^\phi$ by Lemma \ref{lem:leastenergy}. Furthermore, by statements $(a)$ and $(c)$ of Lemma \ref{lem:nehari}, and \cite[Theorem 2.9]{w}, there exists a sequence  $(u_k)$ such that
$$u_k\in D^{2,q'}_0(\mathbb{B})^{\phi},\quad J(u_k) \to c^{\phi}_\infty,\quad \text{and }\quad J'(u_k) \to 0\, \text{ in }\,(D^{2,q'}_0(\mathbb{B})^{\phi})'.$$ 
Now, applying Theorem \ref{thm:profile} we have the following existence alternative: there exists $u\in \mathcal{N}^\phi(\Theta)$ with $J(u)=c_{\infty}^\phi$, either for $\Theta=\mathbb{B}$, or for some half-space $\Theta$, or for $\Theta=\mathbb{R}^N$. As $\mathcal{N}^\phi(\Theta)\subset\mathcal{N}^\phi(\mathbb{R}^N)$ for any $G$-invariant domain $\Theta$, we conclude that, in any case, $J$ attains its minimum on $\mathcal{N}^\phi(\mathbb{R}^N)$.
\end{proof}

Note that, if $\phi\equiv 1$, the solution given by the previous theorem is a least energy $G$-invariant solution. The ground state solution obtained by Lions in \cite{l} is radial, hence, it is $G$-invariant. So Theorem \ref{thm:existence} says nothing new in this case.

The next lemma exhibits surjective homomorphisms which yield different sign-changing minimizers. It was proved in \cite[Lemma 3.2]{cl}. We give the proof here again, to make the symmetries explicit.

\begin{lemma} \label{lem:examples}
For each $j=1,\ldots, \lfloor \frac{N}{4}\rfloor$, there exist a closed subgroup $G_j$ of $O(N)$ and a continuous homomorphism $\phi_j:G_j\to\mathbb{Z}_{2}$ with the following properties:
\begin{itemize}
\item[\emph{(a)}]$\phi:G\to\mathbb{Z}_2$ is surjective,
\item[\emph{(b)}]$G_j$ and $\phi_j$ satisfy $\mathbf{(S1)}$ and $\mathbf{(S2)}$,
\item[\emph{(c)}]If $u,v:\mathbb{R}^N \to \mathbb{R}$ are nontrivial functions, $u$ is $\phi_i$-equivariant, and $v$ is $\phi_j$-equivariant with $i<j$, then $u\neq v$.
\end{itemize}
\end{lemma}

\begin{proof}
Let $\Gamma$ be the group generated by $\{\mathrm{e}^{\mathrm{i}\theta},\varrho: \theta \in [0,2\pi) \}$, acting on $\mathbb{C}^2$ by
$$\mathrm{e}^{\mathrm{i}\theta}(\zeta_1,\zeta_2):=(\mathrm{e}^{\mathrm{i}\theta}\zeta_1,\mathrm{e}^{\mathrm{i}\theta}\zeta_2), \qquad \varrho(\zeta_1,\zeta_2):=(-\bar{\zeta_2},\bar{\zeta_1}), \qquad \text{for} \quad (\zeta_1,\zeta_2) \in \mathbb{C}^2,$$
and let $\phi:\Gamma \to \mathbb{Z}_2$ be the homomorphism given by $\phi(\mathrm{e}^{\mathrm{i}\theta}):=1$ and $\phi(\varrho):=-1$. Note that the $\Gamma$-orbit of a point $z \in \mathbb{C}^2$ is the union of two circles that lie in orthogonal planes if $z \neq 0$, and it is $\{0\}$ if $z=0$.

Set $n=\lfloor \frac{N}{4}\rfloor$, $\Lambda_j:=O(N-4j)$ if $j=1,...,n-1$, and $\Lambda_n:=\{1\}$. Then the $\Lambda_j$-orbit of a point $y \in \mathbb{R}^{N-4j}$ is an $(N-4j-1)$-dimensional sphere if $j=1,\ldots,n-1$, and it is a single point if $j=n$. 

Define $G_j:=\Gamma^j \times \Lambda_j$, acting on $\mathbb{R}^N \equiv (\mathbb{C}^2)^j \times \mathbb{R}^{N-4j}$ by
$$(\gamma_1,\ldots,\gamma_j,\eta)(z_1,\ldots,z_j,y):=(\gamma_1 z_1,\ldots,\gamma_j z_j,\eta y),$$
where $\gamma_i \in \Gamma$, $\eta \in \Lambda_j$, $z_i \in \mathbb{C}^2$, and $y \in \mathbb{R}^{N-4j}$, and let $\phi_j:G_j\rightarrow\mathbb{Z}_{2}$ be the homomorphism
$$\phi_j (\gamma_1,\ldots,\gamma_j,\eta):= \phi(\gamma_1) \cdots \phi(\gamma_j).$$
The $G_j$-orbit of $(z_1,\ldots,z_j,y)$ is the product of orbits 
$$G_j(z_1,\ldots,z_j,y) =  \Gamma z_1 \times \cdots \times \Gamma z_j \times \Lambda_j y.$$ 
Clearly, $\phi_j$ is surjective, and $G_j$ and $\phi_j$ satisfy $\mathbf{(S1)}$ and $\mathbf{(S2)}$ for each $j=1,\ldots,n$.

Now we prove (c). If $u$ is $\phi_i$-equivariant, $v$ is $\phi_j$-equivariant with $i<j$, and $u(x)=v(x) \neq 0$ for some $x=(z_1,\ldots,z_j,y) \in (\mathbb{C}^2)^j \times \mathbb{R}^{N-4j}$, then, as
$$u(z_1,\ldots,\varrho z_j,y) = u(z_1,\ldots,z_j,y) \quad \text{and} \quad v(z_1,\ldots,\varrho  z_j,y) = -v(z_1,\ldots,z_j,y),$$
we have that $u(z_1,\ldots,\varrho z_j,y) \neq v(z_1,\ldots,\varrho z_j,y)$. This proves that $u\neq v$.
\end{proof} 
\medskip

\begin{proof}[Proof of Theorem \ref{thm:main}]
Apply Theorem \ref{thm:existence} to each of the $\phi_j:G_j\to\mathbb{Z}_{2}$ given by Lemma \ref{lem:examples} to obtain pairwise distinct sign-changing solutions $u_1,\ldots,u_n$ to the problem \eqref{eq:prob}. Set $v_i:=-|\Delta u_i|^{q'-2}\Delta u_i$. Since $u_i$ is $\phi_i$-equivariant, $\Delta u_i$ is also $\phi_i$-equivariant and, by Lemma \ref{lem:equivalence}, $(u_i,v_i)$ is a sign-changing solution to the system \eqref{eq:system}.
\end{proof}
\medskip
\begin{remark} \label{rem:n3}
1) At first glance, the symmetries given by Lemma \ref{lem:examples} may seem a bit involved. To illustrate the general shape of a $\phi$-equivariant function we give an explicit example.  Let $f:[0,\infty)\to\R$ be any function and $u:\mathbb{C}^2\to\R$ be given by
\begin{align*}
  u(z_1,z_2)&= f(|z|) \, (|z_1|^2-|z_2|^2)
\end{align*}
for $z=(z_1,z_2)\in \mathbb{C}^2$. Clearly, $u(\mathrm{e}^{\mathrm{i}\theta}z_1,\mathrm{e}^{\mathrm{i}\theta}z_2)=u(z_1,z_2)$ and $u(-\bar{z_2},\bar{z_1})=-u(z_1,z_2)$, i.e., $u$ is $\phi$-equivariant. Note that $u$ is nonradial and changes sign.

2) Theorem \ref{thm:main} is not optimal since, as the proof of Lemma \ref{lem:examples} suggests, there can be other symmetries yielding different solutions.

3) Our approach cannot be used to obtain sign-changing solutions when $N=3$ because no closed subgroup $G$ of $O(3)$ satisfying $(\mathbf{S1})$ and $(\mathbf{S2})$ admits a surjective homomorphism $\phi:G\to\mathbb{Z}_2$, as can be verified by analyzing each subgroup of $O(3)$. The complete list of them is given in \emph{\cite[Section 8]{b}}.
\end{remark}

To close this section we analyze the lack of Möbius invariance of problem \eqref{eq:prob}. 
A \emph{Möbius transformation} $\tau:\mathbb{R}^{N}\cup\{\infty\}\to\mathbb{R}^{N}\cup\{\infty\}$ is a
finite composition of inversions on spheres and reflections on hyperplanes. Recall that the inversion on the sphere $S_{\varrho}(\xi):=\{x\in\mathbb{R}^{N}:|x-\xi|=\varrho\}$, $\xi\in\mathbb{R}^{N}$, $\varrho>0$, is the map $\iota_{\varrho,\xi}$ defined by
$$\iota_{\varrho,\xi}(x):=\xi+\frac{\varrho^2(x-\xi)}{|x-\xi|^2}\quad\text{if }x\neq\xi,\qquad\iota_{\varrho,\xi}(\xi):=\infty,\qquad\iota_{\varrho,\xi}(\infty):=\xi.$$
Since Euclidean isometries are compositions of reflections on hyperplanes, they are Möbius transformations. Dilations $x\mapsto\lambda x$, $\lambda>0$, are also Möbius transformations; see \cite{be}.

If $\tau$ is a Möbius transformation and $u:\mathbb{R}^N\to\mathbb{R}$, we define $u_{\tau}$ by
$$u_{\tau}(x):=|\det \tau'(x)|^{1/p}\,u(\tau(x)).$$
Then, the map $u\mapsto u_{\tau}$ is a linear isometry of $L^p(\mathbb{R}^N)$, i.e., $|u_\tau|_p=|u|_p$ for every $u\in L^p(\mathbb{R}^N)$. Next we investigate, for which values of $q$ this map is also an isometry of $D^{2,q'}(\mathbb{R}^N)$, i.e., $\|u_\tau\|_{q'}=\|u\|_{q'}$, as, for such values, the functional $J$ is Möbius-invariant. The answer is given by the following proposition.

\begin{proposition} \label{prop:no_kelvin}
Let $\iota$ be the inversion on the unit sphere $S_1(0)$. Then, $\|u_\iota\|_{q'}=\|u\|_{q'}$ for every $u\in D^{2,q'}(\mathbb{R}^N)$ if and only if $q\in \{2,2^*\}$.
\end{proposition}
\begin{proof}
As $\iota(x)=\frac{x}{|x|^2}$, the map $u\mapsto u_{\iota}$ is the Kelvin-type transform given by
$$u_\iota(x)=|x|^{-2N/p}\,u\left(\frac{x}{|x|^2}\right).$$
Assume that $\|u_\iota\|_{q'}=\|u\|_{q'}$ for every $u\in D^{2,q'}(\mathbb{R}^N)$. Since $u\mapsto u_{\iota}$ is a continuous linear map, differentiating the identity $\|u_\iota\|_{q'}^{q'}=\|u\|_{q'}^{q'}$ and applying the chain rule we obtain
\begin{align*}
\int_{\mathbb{R}^N}|\Delta u|^{q'-2}\Delta u\Delta v=\int_{\mathbb{R}^N}|\Delta u_\iota|^{q'-2}\Delta u_\iota \Delta v_\iota\qquad\text{for every }u,v\in D^{2,q'}(\mathbb{R}^N).
\end{align*}
Set $A_{a,b}:=\{x\in\mathbb{R}^N:a<|x|<b\}$ and $\alpha:=-\frac{2N}{p}$. Let $u\in\mathcal{C}_c^\infty(\mathbb{R}^N)$ be such that $u(x)=1$ if $x\in A_{1/2,2}$, and define $Lu:=\Delta(|\Delta u|^{q'-2}\Delta u)$. Then $u_\iota(x)=|x|^\alpha$ and $L[u_\iota](x)=C_{N,p,q}|x|^{(q'-1)(\alpha-2)-2}$ for every $x\in A_{1/2,2}$, where
$$C_{N,p,q}=|\alpha(N-2+\alpha)|^{q'-2}(q'-1)(\alpha-2)\left[(q'-1)(\alpha-2)+N-2\right].$$
Therefore,
\begin{align*}
0&=\int_{\mathbb{R}^N}(Lu)\varphi=\int_{\mathbb{R}^N}(L[u_\iota])\varphi_\iota=C_{N,p,q}\int_{\mathbb{R}^N}|x|^{(q'-1)(\alpha-2)-2}\varphi_\iota
\end{align*}
for every $\varphi\in \mathcal{C}_c^\infty(A_{1/2,1})$.  This implies that $C_{N,p,q}=0$. Hence, either $\frac{2N}{p}=N-2$, or $(q'-1)(\frac{2N}{p}+2)=N-2$. Recall that $\frac{N}{p}+\frac{N}{q}=N-2$ and $\frac{N}{q}+\frac{N}{q'}=N$. Thus, $\frac{2N}{p}=N-2$ iff $2^*=p=q$, and $(q'-1)(\frac{2N}{p}+2)=N-2$ iff $q=2$, as claimed.

The opposite statement is the Kelvin-invariance for the Yamabe equation \eqref{eq:yamabe} and the Paneitz equation \eqref{eq:paneitz}, which is well known; see \cite{bsw,d}.
\end{proof}

\appendix

\section{The weak limits are solutions}

In \cite{mw, cl} a truncation is used to show that the weak limits $u$ and $w$ in the proof of Theorem \ref{thm:profile} are solutions of a limit problem. Truncations are commonly used in the study of the $p$-Laplacian, but they do not work well in the higher-order setting because gradient discontinuities prevent the truncated function from being twice weakly differentiable. Here we give a different argument, that can also be applied to more general problems, like those described in the introduction.

Let $\Theta$ be a $G$-invariant smooth domain in $\mathbb{R}^N$, not necessarily bounded, and let $C^\infty_c(\Theta)^\phi$ and $C^\infty_c(\Theta)^G$ denote the spaces of functions in $C^\infty_c(\Theta)$ which are $\phi$-equivariant and $G$-invariant respectively. The main result in this appendix is the following one.

\begin{proposition} \label{prop:A}
Let $v_k,v\in D^{2,q'}(\R^N)^\phi$ be such that $v_k\rightharpoonup v$ weakly in $D^{2,q'}(\R^N)$.  Assume that
\begin{equation*}
  \lim_{k\to\infty}  J'(v_k)\varphi = 0\qquad  \text{ and }\qquad 
  \lim_{k\to\infty}J'(v_k)[\psi (v_k-v)]=0
\end{equation*}
for every $\varphi\in C^\infty_c(\Theta)^\phi$ and $\psi\in C^\infty_c(\Theta)^G$.  Then $J'(v)\varphi=0$ for all $\varphi\in C^\infty_c(\Theta)^\phi$.
\end{proposition}

We start with the following lemmas. For a set $U\subset \R^N$, we use $|U|$ to denote its Lebesgue measure.

\begin{lemma}\label{lem:quotient}
  Let $U\subset\R^N$ be a measurable bounded set with $|U|>0$, let $(f_k)$ be a sequence of nonnegative functions which is bounded in $L^1(U)$, and let $\alpha>0$. Then, there exists $\kappa>0$ such that
  \begin{align*}
    \int_U \frac{1}{(f_k+1)^\alpha}>\kappa\qquad \text{for all }\;k\in\N.
  \end{align*}
\end{lemma}
\begin{proof}
Let $C\geq \int_U f_k$ for all $k\in\N$. Fix $n\in\N$ such that $(n-1)|U|>C$ and, for each $k\in\N$, set 
$$ U_k:=\{x\in U:f_k(x)\geq n-1\}=\left\{x\in U:\frac{1}{f_k(x)+1}\leq \frac{1}{n}\right\}.$$
Since $f_k\geq 0$, we have that
  \begin{align*}
C\geq \int_U f_k \geq \int_{U_k} f_k\geq (n-1)|U_k| \qquad\text{for all }k\in\N.
  \end{align*}
Then $|U_k|\leq \frac{C}{n-1}$ and therefore,
  \begin{align*}
\int_U \frac{1}{(f_k+1)^\alpha}
&=\int_{U_k} \frac{1}{(f_k+1)^\alpha}+\int_{U\smallsetminus U_k} \frac{1}{(f_k+1)^\alpha}\\
&\geq \frac{1}{n^\alpha} |U\smallsetminus U_k|\geq \frac{1}{n^\alpha}\left(|U|-\frac{C}{n-1}\right)>0,
  \end{align*}
as claimed.
\end{proof}

\begin{lemma}\label{lem:pointwise_convergence}
Let $v_k,v\in D^{2,q'}(\R^N)^\phi$ be such that $v_k\rightharpoonup v$ weakly in $D^{2,q'}(\R^N)$.  Assume that
  \begin{align}\label{A:0:eq}
\lim_{k\to\infty}\int_{\Theta} \psi |\Delta v_k|^{q'-2}\Delta v_k \Delta(v_k-v)=0
\qquad \text{for every $\psi\in C^\infty_c(\Theta)^G.$}
  \end{align}
Then, after passing to a subsequence, $\Delta v_k\to\Delta v$ a.e. in $\Theta$.
\end{lemma}

\begin{proof}
As shown in \cite{li2}, there is a constant $C_0>0$, which depends only on $q'$, such that, for every $s,t\in\mathbb{R}$,
\begin{equation}\label{A:1:eq}
 (|s|^{q'-2}s - |t|^{q'-2}t)(s-t)\geq
  \begin{cases}
    C_0|s-t|^{q'} & \text{ if } q'\geq 2,\smallskip \\
    C_0\frac{|s-t|^2}{(|s|^{q'}+|t|^{q'}+1)^{2-q'}} & \text{ if } 1<q'<2.
  \end{cases}
\end{equation}
Let $v_k$ and $v$ as in the statement and set $f_k:=|\Delta v_k|^{q'-2}\Delta v_k-|\Delta v|^{q'-2}\Delta v$ and 
\begin{equation*} 
h_k:=
  \begin{cases}
    C_0|\Delta v_k-\Delta v|^{q'} & \text{ if } q'\geq 2,\smallskip \\	
    C_0\frac{|\Delta v_k-\Delta v|^2}{(|\Delta v_k|^{q'}+|\Delta v|^{q'}+1)^{2-q'}} & \text{ if } 1<q'<2.
  \end{cases}
\end{equation*}
From \eqref{A:1:eq}, \eqref{A:0:eq}, and the fact that $v_k\rightharpoonup v$ weakly in $D^{2,q'}(\R^N)$ we get that
\begin{equation} \label{eq:null}
0\leq \lim_{k\to\infty}\int_{\Theta}h_k\psi\leq
\lim_{k\to\infty}\int_{\Theta}f_k\psi=0
\end{equation}
for every nonnegative $\psi\in C^\infty_c(\Theta)^G.$  If $q'\geq 2$, this immediately implies that $\Delta v_k\to\Delta v$ a.e. in $\Theta$. If $1<q'<2$, we argue by contradiction. Assume that, after passing to a subsequence, there is a compact set $K\subset\Theta$ with positive measure, and a constant $\mu>0$, such that 
  \begin{align*}
|\Delta (v_k-v)(x)|>\mu\qquad \text{for all }\;x\in K,\;k\in\N.     
  \end{align*}
Fix $\psi\in C^\infty_c(\Theta)^G$ nonnegative with $\psi(x)=1$ for every $x\in K$. Then, \eqref{eq:null} implies that
$$\lim_{k\to\infty}\int_K\frac{1}{(|\Delta v_k|^{q'}+|\Delta v|^{q'}+1)^{2-q'}}=0,$$
contradicting Lemma \ref{lem:quotient}. The proof is complete.
\end{proof}
\medskip

\begin{proof}[Proof of Proposition \ref{prop:A}]
We show first that, after passing to a subsequence, $\Delta v_k\to \Delta v$ a.e. in $\Theta$. Let $\psi\in C^\infty_c(\Theta)^G$. To simplify notation, we write
$$f_k:=|\Delta v_k|^{q'-2}\Delta v_k\qquad\text{and}\qquad 
h_k:= |v_k|^{p-2}v_k.$$
Note that $(f_k)$ is bounded in $L^q(\mathbb{R}^N)$, $h_k$ is bounded in $L^{p'}(\mathbb{R}^N)$, and 
\begin{align} \label{eq:1}
\left|\int_{\mathbb{R}^N}f_k\psi \Delta(v_k-v)\right| \leq &
\left|\int_{\mathbb{R}^N} f_k\Delta(\psi(v_k-v))\right|
+\left|\int_{\mathbb{R}^N} f_k(\Delta\psi) (v_k-v)\right|\nonumber\\
&+2\left|\int_{\mathbb{R}^N} f_k\nabla \psi\cdot \nabla (v_k-v)\right|.
\end{align}
Let $\gamma>0$. We choose mollifiers $\eta_\varrho\in C^\infty_c(\mathbb{R}^N)$ with 
$\eta_\varrho\geq 0$, $\mathrm{supp}(\eta_\varrho)\subset \overline{B_\varrho(0)}$ and $\int_{\mathbb{R}^N}\eta_\varrho=1$. Then, since $J'(v_k)[\psi (v_k-v)]=o(1)$, $\eta_\varrho*(v_k-v)\to v_k-v$ in $L^p(\mathbb{R}^N)$
as $\varrho\to 0$, and $(h_k)$ is bounded in $L^{p'}(\mathbb{R}^N)$, we may fix $\varrho>0$ such that, for $k$ large enough,
\begin{align*}
\left|\int_{\mathbb{R}^N} f_k\Delta(\psi(v_k-v))\right|&
=\left|\int_{\mathbb{R}^N} h_k\psi(v_k-v)\right|+o(1)\\
&\leq \left|\int_{\mathbb{R}^N} h_k\psi(\eta_\varrho*(v_k-v))\right|+\gamma +o(1)\\
&\leq C\left(\int_K|\eta_\varrho*(v_k-v)|^p\right)^{1/p}+\gamma+o(1),
\end{align*}
where $K:=\mathrm{supp}(\psi)$. Now, for any $r\in[1,p]$ and $x\in K$, Hölder's inequality yields
$$|\eta_\varrho*(v_k-v)(x)|\leq |\eta_\varrho|_{r'}\left(\int_{B_\varrho(K)}|v_k-v|^r\right)^{1/r}.$$
As $v_k\rightharpoonup v$ weakly in $D^{2,q'}(\R^N)$, the Rellich-Kondrashov theorem asserts that $v_k\to v$ in $L_{\mathrm{loc}}^r(\R^N)$ for every $r\in [1,p)$. Hence, $\eta_\varrho*(v_k-v)\to 0$ pointwise in $K$. Moreover, taking $r=p$, we get that $|(\eta_\varrho*(v_k-v))(x)|\leq C$ for every $x\in K$. So, the dominated convergence theorem yields
$$\int_K|\eta_\varrho*(v_k-v)|^p=o(1)$$
and, consequently,
\begin{equation} \label{eq:2}
\left|\int_{\mathbb{R}^N} f_k\psi(v_k-v)\right|\leq o(1)+\gamma\qquad\text{for every}\;\gamma>0.
\end{equation}
Furthermore, since $(f_k)$ is bounded in $L^q(\mathbb{R}^N)$, we have that
\begin{equation} \label{eq:3}
\left|\int_{\mathbb{R}^N} f_k(\Delta\psi) (v_k-v)\right|\leq C|f_k|_q\left(\int_K|v_k-v|^{q'}\right)^{1/q'}=o(1),
\end{equation}
and
\begin{equation} \label{eq:4}
\left|\int_{\mathbb{R}^N} f_k\nabla \psi\cdot \nabla (v_k-v)\right|\leq C|f_k|_q\left(\int_K|\nabla(v_k-v)|^{q'}\right)^{1/q'}=o(1),
\end{equation}
because, by the Rellich-Kondrashov theorem, $|\nabla(v_k-v)|\to 0$ in $L_{\mathrm{loc}}^r(\R^N)$ for every $r\in [1,(q')^*)$ with $(q')^*=\frac{Nq'}{N-q'}>q'$.
From \eqref{eq:1}, \eqref{eq:2}, \eqref{eq:3}, and \eqref{eq:4}, we derive
$$ \lim_{k\to\infty}\left|\int_{\mathbb{R}^N}f_k\psi \Delta(v_k-v)\right|=0.$$
Thus, by Lemma \ref{lem:pointwise_convergence}, $\Delta v_k\to \Delta v$ a.e. in $\Theta$, as claimed.

Let $\varphi\in C^\infty_c(\Theta)^\phi$, $X:=\mathrm{supp}(\varphi)$,
$f:=|\Delta v|^{q'-2}\Delta v$, and $h:=|v|^{p-2}v$.  As $(f_k-f)\Delta\varphi\to 0$ a.e. in $\Theta$, Egorov's theorem asserts that, for any $\gamma>0$, there is a subset $Z_\gamma$ of $X$ with $|Z_\gamma|<\gamma$ such that $(f_k-f)\Delta\varphi\to 0$ uniformly in $X\smallsetminus Z_\gamma$. Therefore,
\begin{align*}
\left|\int_{\Theta} (f_k-f)\Delta\varphi\right|\leq \left|\int_{Z_\gamma} (f_k-f)\Delta\varphi\right|+\left|\int_{X\smallsetminus Z_\gamma} (f_k-f)\Delta\varphi\right| \leq C\gamma + o(1).
\end{align*}
A similar argument shows that
\begin{align*}
\left|\int_{\Theta} (h_k-h)\,\varphi\right|\leq C\gamma + o(1).
\end{align*}
Since $\gamma$ is arbitrary, we conclude that
\begin{align*}
\lim_{k\to\infty}|J'(v_k)\varphi-J'(v)\varphi|=\lim_{k\to\infty}\left|\int_{\Theta}(f_k-f)\Delta\varphi-(h_k-h)\,\varphi\right| =0.
\end{align*}
Therefore, $J'(v)\varphi=0$ for all $\varphi\in C^\infty_c(\Theta)^\phi$, as claimed.
\end{proof}

 \vspace{15pt}

\begin{flushleft}
\textbf{Mónica Clapp}\\
Instituto de Matemáticas\\
Universidad Nacional Autónoma de México\\
Circuito Exterior, Ciudad Universitaria\\
04510 Coyoacán, CDMX\\
Mexico\\
\texttt{monica.clapp@im.unam.mx} \vspace{10pt}

\textbf{Alberto Saldaña}\\
Instituto de Matemáticas\\
Universidad Nacional Autónoma de México\\
Circuito Exterior, Ciudad Universitaria\\
04510 Coyoacán, CDMX\\
Mexico\\
\texttt{alberto.saldana@im.unam.mx}
\end{flushleft}

\end{document}